\newif\ifTwoColumn
\newif\ifTechReport
\TwoColumntrue

\ifTwoColumn
\documentclass[letterpaper, 10pt, conference]{ieeeconf}
\usepackage{flushend}
\IEEEoverridecommandlockouts
\newtheorem{remark}{Remark}
\else
\ifTechReport
\documentclass[11pt,onecolumn,a4paper]{IEEEtran}
\usepackage{geometry}
\usepackage{amsthm}%
\usepackage{parskip}
\makeatletter
\def\thm@space@setup{%
	\thm@preskip=\parskip \thm@postskip=0pt
}
\makeatother

\newtheorem*{remark}{Remark}
\else
\documentclass[12pt,draftcls,onecolumn,letterpaper]{IEEEtran}
\usepackage{amsthm}
\IEEEoverridecommandlockouts

\fi
\fi

\usepackage[utf8]{inputenc}
\usepackage{graphicx}
\usepackage{amsfonts}
\usepackage{amsmath}
\usepackage{nicematrix}
\usepackage{float}
\usepackage{amssymb}
\usepackage[colorlinks,linkcolor=blue]{hyperref}
\usepackage{multirow}
\usepackage{float}
\usepackage{bm}
\usepackage{graphicx}
\usepackage{stfloats}
\usepackage{color}
\usepackage{subfigure}
\usepackage{algorithm}
\usepackage{algorithmic}
\usepackage{url}
\usepackage[utf8]{inputenc}
\usepackage[english]{babel}
\usepackage{verbatim}
\usepackage{ulem}
\usepackage[center]{caption}
\usepackage{placeins}

\usepackage{xcolor,calc}

\newtheorem{theorem}{Theorem}[section]

\newtheorem{assumption}[theorem]{Assumption}

\allowdisplaybreaks[4]
\makeatletter
\newenvironment{breakablealgorithm}
{
	\begin{center}
		\refstepcounter{algorithm}
		\hrule height.8pt depth0pt \kern2pt
		\renewcommand{\caption}[2][\relax]{
			{\raggedright\textbf{\ALG@name~\thealgorithm} ##2\par}%
			\ifx\relax##1\relax 
			\addcontentsline{loa}{algorithm}{\protect\numberline{\thealgorithm}##2}%
			\else 
			\addcontentsline{loa}{algorithm}{\protect\numberline{\thealgorithm}##1}%
			\fi
			\kern2pt\hrule\kern2pt
		}
	}{
		\kern2pt\hrule\relax
	\end{center}
}
\makeatother

\title{\LARGE \bf A Unified Early Termination Technique for Primal-dual Algorithms in Mixed Integer Conic Programming}
\author{Yuwen Chen$^*$, Catherine Ning, Paul Goulart
\thanks{$^*$Corresponding author: {\tt\footnotesize{yuwen.chen@eng.ox.ac.uk}}. The authors are with the Department of Engineering Science, University of Oxford, UK.}
} 

\begin{document}

\maketitle

\begin{abstract}
We propose an early termination technique for mixed integer conic programming for use within branch-and-bound based solvers. Our approach generalizes previous early termination results for ADMM-based solvers to a broader class of primal-dual algorithms, including both operator splitting methods and interior point methods. The complexity for checking early termination is $O(n)$ for each termination check assuming a bounded problem domain.  We show that this domain restriction can be relaxed for problems whose data satisfies a simple rank condition, in which case each check requires an $O(n^2)$ solve using a linear system that must be factored only once at the root node. We further show how this approach can be used in hybrid model predictive control as long as system inputs are bounded. Numerical results show that our method leads to a moderate reduction in the total iterations required for branch-and-bound conic solvers with interior-point based subsolvers.
\end{abstract}

\section{Introduction}

\subsection{Literature Review}
Mixed integer conic programming (MICP) is a powerful tool for modelling many real-world applications, e.g. hybrid model predictive control \cite{Bemporad99}, portfolio optimization \cite{portfolio}, power electronics \cite{Stellato17} and robust truss topology \cite{Yonekura10}. The branch-and-bound (B\&B) method is the most commonly used technique for the search of an optimal solution in MIP solvers. B\&B algorithms must solve a sequence of relaxed convex subproblems, and the number of such problems increases exponentially w.r.t.\ the number of integer variables. 

Many techniques have been developed to speed up MIP computation. Cutting plane methods are widely used to reduce the problem search space and can reduce significantly the number of nodes that a B\&B must visit. Presolve \cite{Achterberg20} can be regarded as a collection of preprocessing methods before solving a MIP, including bound strengthening, coefficient strengthening, constraint reduction and conflict analysis. In addition to presolving a MIP one also apply many heuristic methods to accelerate the computation. Most acceleration methods can be broadly classified into two types, start and improvement heuristics \cite{Berthold06}, both of which are crucial for pruning nodes in B\&B algorithms. Start heuristics aim to find a feasible solution as early as possible when the B\&B algorithm starts, e.g.\ feasibility pump \cite{Berthold19}. On the other hand, improvement heuristics search for feasible points of better objective value based on information from feasible points already obtained, e.g.\ RINS \cite{Danna05} and the crossover method~\cite{Rothberg07}.

Pruning is usually an effective method to reduce the total number of nodes to be solved in B\&B. Suppose $U$ is the upper bound corresponding to the value of the best integer feasible solution so far. After updating the upper bound $U$ with a new integer feasible point, one can prune any unevaluated nodes that are known to have an optimal value, or any lower bound thereof, greater than $U$. Consequently, if a dual feasible point of a relaxed problem within a B\&B search can be generated prior to convergence with its dual objective already larger than the current upper bound $U$, then one can stop the node computation immediately before solving it to optimality. This is called \textit{early termination} and has been implemented in dual feasible algorithms~\cite{Fletcher98,Naik17,Axehill06,Buchheim16}. 

At the core of any B\&B method is an optimization algorithm for solving convex problems. Many state-of-the-art conic optimization algorithms are primal-dual methods and most of them can be classified into two types: one is a second-order method called interior point method (IPM)~\cite{Nesterov94} and another one is a first-order method called operator splitting method (OSM)~\cite{Eckstein92}. Both of them start from an infeasible initial point, and attain a feasible point when the algorithm converges to a global optimum or generate a certificate of infeasibility otherwise. This makes early termination difficult since primal-dual methods do not typically reach a dual feasible point until the algorithm converges at optimality. Recently, \cite{Liang21} proposes a heuristic method generating a dual feasible point for a specialized primal-dual IPM, but the feasibility of dual iterates is still not theoretically guaranteed and it applies only to mixed-integer quadratic programming.

\subsection{Contributions and Organization}
In this paper we generalize an early termination strategy for mixed integer conic programming, initially proposed for ADMM~\cite{Chen22}, to any primal-dual optimization method. We develop efficient methods to find a dual feasible point for early termination at each iteration. We relax the boundedness assumption in~\cite{Chen22} to a more general rank condition on the problem data which is applicable for many real-world scenarios. We propose a simple correction step that costs $O(n)$ flops for bounded problems, and a more general optimization-based one costing $O(n^2)$ flops at each iteration once we obtain a factorization at the start of a MICP. Both costs are relatively small compared to the factorization time $O(n^3)$ per iteration in IPMs and no worse than the per iteration cost of OSMs.

Section 2 provides background on conic optimization. Section 3 presents our early termination strategy for mixed integer conic programming and describes how to implement it in both OSMs and IPMs. Numerical results are shown in Section 4 and conclusions summarized in Section 5.

\subsection{Notation}
We denote the $n \times n$ symmetric matrices by $\mathbb{S}^n$, and the set of positive semidefinite matrices $\mathbb{S}^n_+$. We denote $|\mathbb{I}|$ as the number of elements in the discrete set $\mathbb{I}$. The norm $\|\cdot \|$ is the Euclidean norm. The projection $\Pi_{\mathcal{C}}(x)$ denotes projecting $x \in \mathbb{R}^n$ onto set $\mathcal{C}$. The support function of $\mathcal{C}$ is 
\begin{align*}
	\mathcal{\sigma}_{\mathcal{C}}(x) := \sup_{y \in \mathcal{C}} \langle x, y \rangle.
\end{align*}
We denote the dual cone $\mathcal{K}^*$ and polar cone $K^{\circ}$ of a convex cone $\mathcal{K}$ by
	$
	\mathcal{K}^* := \{ y \in \mathbb{R}^n \ | \ \sup_{x \in \mathcal{K}} \langle x, y \rangle \ge 0 \},
$
and 
	$
	\mathcal{K}^{\circ} := \{ y \in \mathbb{R}^n \ | \ \sup_{x \in \mathcal{K}} \langle x, y \rangle \le 0 \},
	$
	respectively.

\section{Background}
\subsection{Problem Formulation}
We will consider MICPs in the general form:
\begin{align}
	\begin{aligned}
		\min \quad & \frac{1}{2} x^\top P x + q^\top x\\
		\text{s.t.} \quad & Gx = h\\
		& Ax + s = b,  \ s \in \mathcal{K}, \\
		& \bar{l} \le x \le \bar{u}, \ x_{\mathbb{I}} \in \mathcal{Z},
	\end{aligned}\label{MICP-form}
\end{align}
where $G \in \mathbb{R}^{p \times n}, A \in \mathbb{R}^{m \times n}, h \in \mathbb{R}^{p}, b \in \mathbb{R}^{m}$ and $\mathcal{K}$ is a proper cone. The vector $x \in \mathbb{R}^n$ is the decision variable with interval bounds defined by $\bar{l},\bar{u} \in \mathbb{R}^n$, and $\mathbb{I}$ denotes the entries of $x$ constrained to a finite integer set $\mathcal{Z}$. The objective function is convex quadratic with symmetric positive semidefinite $P \in \mathbb{S}_+^n$ and vector $q \in \mathbb{R}^n$. We denote the continuous relaxation of \eqref{MICP-form} as
\begin{align}
	\begin{aligned}
		\min \quad & \frac{1}{2} x^\top P x + q^\top x\\
		\text{s.t.} \quad & Gx = h\\
		& Ax + s = b,  \ s \in \mathcal{K}, \\
		& l \le x \le u,
	\end{aligned}\qquad \text{CP}(l, u) \label{MICP-relaxation}
\end{align}
where the integer relaxation of $\mathcal{Z}$ is incorporated into the box constraint $\bar{l} \le l \le x \le u \le \bar{u}$. 

\subsection{Dual Form for Operator Splitting Methods}
Following~\cite{Chen22}, the dual of the continuous relaxation~\eqref{MICP-relaxation} is
\begin{align}
	\begin{aligned}
		\max_{x,y, y_b,z} & -\frac{1}{2}x^\top P x - h^\top z + b^\top y - \mathcal{\sigma}_{[l,u]}(y_b) \\
		\text{s.t} \quad & Px + q + G^\top z - A^\top y + y_b = 0, \\
		& x \in \mathbb{R}^n, y \in \mathcal{K}^{\circ}, y_b \in \mathbb{R}^n, z \in \mathbb{R}^p, \label{ADMM-dual}
	\end{aligned}
\end{align}
where the support function $\mathcal{\sigma}_{[l,u]}(y_b)$ is explicit, i.e.
\begin{align}
	\begin{aligned}
		\mathcal{\sigma}_{[l,u]}(y_b) = u^\top y_b^+ + l^\top y_b^- \label{box-support},
	\end{aligned}
\end{align}
 where $y_b^+ = \max \{y_b,0\}, \quad y_b^- = \min \{y_b, 0 \},$, which is suitable to generate a correction for early termination of any MIP based on an operator splitting solver, e.g. OSQP~\cite{Stellato18} and PDHG~\cite{Applegate21}. 
 
 \subsection{Dual Form for Primal-Dual Interior-Point Methods}
 
 For IPMs that rely on logarithmically homogeneous self-concordant barrier (LHSCB) functions~\cite{Nesterov94}, there is no standard explicit barrier function for box constraints. We instead reformulate the box constraint $l \le x \le u$ into two nonnegative inequalities $x \ge l, x \le u$ that have well-defined barrier functions, and obtain another dual formulation as:
\begin{align}
	\begin{aligned}
		\max_{x,y, y_+, y_-,z} & -\frac{1}{2}x^\top P x - h^\top z - b^\top y - u^\top y_+ + l^\top y_- \\
		\text{s.t} \quad & Px + q + G^\top z + A^\top y + y_+ - y_- = 0 \\
		&x \in \mathbb{R}^n, y \in \mathcal{K}^*, \ y_- \ge 0, y_+ \ge 0, z \in \mathbb{R}^p, \label{IPM-dual}
	\end{aligned}
\end{align}
where $\mathcal{K}^* = -\mathcal{K}^{\circ}$ for a proper cone $\mathcal{K}$. If we define $y_b := y_+ - y_-$ for~\eqref{IPM-dual}, then we find that the dual form for IPMs~\eqref{IPM-dual} is the same as its counterpart~\eqref{ADMM-dual} for OSMs. We can therefore design a unified dual correction mechanism for both IPMs and OSMs, which we describe in Section~\ref{section:correction}. 

The primal-dual IPM  typically requires factorization of a matrix in the~form
\begin{align}
	\begin{aligned}
		K := \left[\begin{array}{cccc}
			P & G^\top & A^\top \\
			G & 0 & 0\\
			A & 0 & -H^k \\
		\end{array}\right] 
	\end{aligned} \label{IPM-KKT}
\end{align}
to compute the search direction for every iteration $k$, where $H^k$ is the scaling matrix depending on the choice of cones but it is always positive semidefinite. By adding small perturbation to diagonals of $K$, the matrix can become quasi-definite and be factorized by LDL decomposition with complexity $O((n+p+m)^3)$~\cite{ECOS,Clarabel}. It can always generate a sequence $(x^k, s^k, z^k, y^k, y_b^k)$ such that $s^k \in \mathcal{K}$ and $y^k \in \mathcal{K}^*$, which is the same as in OSMs.

\subsection{Branch and Bound}
The B\&B method computes the optimal solution $x$ in~\eqref{MICP-form} by exploring different integer combinations in a tree. It repeatedly branches on some entries of $x$ in the integer index set $\mathbb{I}$ and solves the continuous relaxation \eqref{MICP-relaxation} until a global optimizer is found. Meanwhile, B\&B always maintain a globally upper bound $U$, which corresponds to the best integer feasible solution of~\eqref{MICP-form} found so far. The upper bound is very useful to prune unsolved nodes and save computation. We would exploit it in our early termination strategy later. 

\section{Early Termination for Primal-dual Algorithms}\label{section:correction}
In this section we first review the early termination technique we proposed in~\cite{Chen22}, arguing that it is also applicable in other OSMs, and then tailor it for primal-dual IPMs. We also relax our boundedness assumption and improve the optimization-based correction discussed in~\cite{Chen22} and discuss how we can apply it into hybrid MPC problems. 

The key for our proposed early termination is to utilize the current dual iterate which has a conic feasible $y^k$ from a primal-dual algorithm, either an OSM or an IPM, and then remove linear dual residuals by adding compensation to unconstrained dual variables. We thereby obtain a dual feasible solution for~\eqref{ADMM-dual} or~\eqref{IPM-dual} and generate the corresponding dual cost for early termination. To ensure our early termination always works, we make the following boundedness assumption as in~\cite{Chen22}.   
\begin{assumption}
	The domain of $x$ in the MIP relaxation~\eqref{MICP-relaxation} is bounded, i.e. $l, u \in \mathbb{R}^n$ are both finite. \label{assumption-1}
\end{assumption}
The assumption is applicable for many real world scenarios, e.g. $x$ is an 0-1 switching signal or subjected to some physical limitations, like in some QP problems where $\|x\|$ is bounded.  We will show how to relax this assumption in Section~\ref{subsection:optimization-based-correction-partial}.
\subsection{Correction for OSMs} \label{subsection-correction-splitting}
ADMM can generate iterates $y^k \in \mathcal{K}^{\circ}$, $\forall k \ge 0$ in~\cite{Garstka21}. For any dual iterates $(x^k,y^k,y_b^k,z^k)$ generated by ADMM, we can offset the residual
\begin{align}
	r^k := Px^k + q + G^\top z^k - A^\top y^k + y_b^k
\end{align}
by setting $\Delta y_b^k = -r^k$ so that $(x^k,y^k,y_b^k + \Delta y_b^k,z^k)$ is a dual feasible point for~\eqref{ADMM-dual}, which is suitable for the early termination technique proposed in~\cite{Chen22}. A useful property of ADMM is that it always generates a $y^k$ satisfying conic constraints $\mathcal{K}^{\circ}$. However, such a property can be generalized to any OSM because we always tackle a conic constraint $s \in \mathcal{K}$ by either the projection to the polar cone $\mathcal{K}^{\circ}$, i.e. $\Pi_{\mathcal{K}^{\circ}}(v^k)$, or the projection to $\mathcal{K}$, i.e. $\Pi_{\mathcal{K}}(v^k)$. The former is what we want for early termination directly, like Google's primal-dual hybrid gradient (PDHG) solver~\cite{Applegate21}. For the latter, Due to the Moreau decomposition [Section 2.5~\cite{Parikh14}], \begin{align}
	v = \Pi_{\mathcal{K}}(v) + \Pi_{\mathcal{K}^{\circ}}(v), \ \forall v,
\end{align}
we can generate an "equivalent" dual iterate $(I - \Pi_{\mathcal{K}})(v) \in \mathcal{K}^{\circ}$, that is $y^k$ we obtained in ADMM~\cite{Chen22}. Therefore, we claim the early termination we proposed in~\cite{Chen22} can be implemented for any OSM within a B\&B solver.

\subsection{Correction for primal-dual IPMs}\label{subsection-correction-IPM}
Note that the main idea behind our correction strategy is to make the iterate $(x^k,y^k,y_+^k,y_-^k,z^k)$ dual feasible. A similar idea can be applied to primal-dual IPMs, which also generate dual-feasible conic iterates $y^k$ for every iteration $k$. Suppose we define $\Delta y_b := \Delta y_{+} - \Delta y_{-}$ with $\Delta y_{+}, \Delta y_{-} \ge 0$ for the IPM dual formulation~\eqref{IPM-dual}. We can verify $\Delta y_b$ is an unconstrained variable for the dual correction. If we only make corrections on $y_-$ and $y_+$, leaving other variables fixed, then the change of dual cost in~\eqref{IPM-dual} becomes
\begin{align}
\begin{aligned}
	-\Delta y_+^\top u +\Delta y_-^\top l & = \Delta y_+^\top(l-u) + (\Delta y_- - \Delta y_+)^\top l\\
	& = \Delta y_+^\top(l-u) - \Delta y_b^\top l.
\end{aligned} \label{box-correction}
\end{align}
Note that we have $ \Delta y_+^\top(l-u) \le 0$ due to $ \Delta y_+ \ge 0$, $l-u \le 0$. Meanwhile, the linear residual is
\begin{align}
	r^k := Px^k + q + G^\top z^k + A^\top y^k + y_+^k - y_-^k \label{linear-residual}
\end{align} 
before the correction. To maximize the dual objective in~\eqref{IPM-dual} given $\Delta y_b^k = -r^k$, we set $\Delta y_+^k, \Delta y_-^k$ as
\begin{align}
	\Delta y_+^k = \max \{0, \Delta y_b^k\}, \quad \Delta y_-^k = \Delta y_+^k - \Delta y_b^k.\label{correction-delta-y}
\end{align}
Hence, $(x^k,y^k,y_+^k + \Delta y_+^k,y_-^k + \Delta y_-^k, z^k)$ is a dual feasible point and we can enable early termination checking via~\eqref{IPM-dual}.

\subsection{Optimization-based correction}\label{subsection:optimization-based-correction-partial}
In Section~\ref{subsection-correction-splitting} and~\ref{subsection-correction-IPM}, we need the correction for each entry of $y_+^k,y_-^k$ ensuring dual feasibility, and that is why Assumption~\ref{assumption-1} comes into play. However, the core of early termination is to offset the linear residual $r^k$ in~\eqref{linear-residual} via corrections on unconstrained dual variables, which means we can exploit other dual variables beyond box constraints. Suppose we are going to utilize unconstrained dual variables $x, y_b, z$ for correction in early termination, Assumption~\ref{assumption-1} can be generalized to the assumption below.
\begin{assumption}
	$[P, I_{\mathcal{B}}^\top, G^\top]$ is of rank $n$, i.e. full row-rank, where $\mathcal{B}$ is the set of entries that have explicit bounded constraints $l_{\mathcal{B}} \le x_{\mathcal{B}} \le u_{\mathcal{B}}$ and $I_{\mathcal{B}}$ is the incidence matrix from the span of $x$ to entries in $\mathcal{B}$, i.e. $x_{\mathcal{B}} = I_{\mathcal{B}}x$.\label{assumption-2}
\end{assumption}  
Under Assumption~\ref{assumption-2}, we can always generate a dual feasible correction $(\Delta x^k, \Delta y_b^k, \Delta z^k)$ since the following linear system always has a solution,
\begin{align}
	P \Delta x^k + I_{\mathcal{B}}^\top \Delta y_{\mathcal{B}}^k + G^\top \Delta z^k = -r^k. \label{generalized-linear-solve}
\end{align}
It is also a generalization for setting $\Delta y_b^k = -r^k$ discussed in Section~\ref{subsection-correction-IPM}, which is useful if some entries of $l,u$ for box constraints are infinite or the difference $u-l$ is so large that the corrected dual cost is excessively sensitive to the correction $\Delta y_b^k$. 

Due to the existence of different coefficients for the support function $\sigma_{[l,u]}(y_b)$ in~\eqref{ADMM-dual} or $- u^\top y_+ + l^\top y_-$ in~\eqref{IPM-dual}, we divide the optimization-based correction into two steps. For the first step, we solve the optimization problem
\begin{align}
	\begin{aligned}
		\min_{\Delta x^k, \Delta z^k, \Delta y_{\mathcal{B}}^k} & \frac{1}{2}\Delta x^{k \top} P\Delta x^k + (Px^k)^\top \Delta x^k + h^\top \Delta z^k \\
		& \quad  + \frac{\eta}{2}\|\Delta y_{\mathcal{B}}^k\|^2 + \frac{\gamma}{2}\|\Delta z^k\|^2 \\
		\text{s.t.} \quad & P \Delta x^k + I_{\mathcal{B}}^\top \Delta y_{\mathcal{B}}^k + G^\top \Delta z^k = -r^k,
	\end{aligned}\label{correction-opt-partial}
\end{align}
which produces a correction $(\Delta x^k,\Delta y_{\mathcal{B}}^k,\Delta z^k)$ while maximizing the corrected dual cost w.r.t. $\Delta x^k, \Delta z^k$ with regularizations for $\Delta y_{\mathcal{B}}^k, \Delta z^k$. The corresponding KKT condition of~\eqref{correction-opt-partial} is
\begin{align}
	\begin{bmatrix}
		P &  I_{\mathcal{B}}^\top & G^\top \\
		I_{\mathcal{B}} & -\eta I & 0 \\
		G & 0 & -\gamma I
	\end{bmatrix}
	\begin{bmatrix}
		\Delta x^k\\
		\Delta y_{\mathcal{B}}^k\\
		\Delta z^k
	\end{bmatrix}
	= \begin{bmatrix}
		-r^k\\
		-I_{\mathcal{B}} x^k\\
		h - Gx^k
	\end{bmatrix}\label{correction-partial-linear-system}
\end{align}
if we set $\lambda^k = x^k + \Delta x^k$. The matrix on the left-hand side does not depend on the active node, and hence only needs to be factored once at the initialization of a MIP solver and can be reused later for any node's computation. Meanwhile, solving~\eqref{correction-partial-linear-system} is computationally efficient compared to the factorization step of an IPM in every iteration (compare \eqref{IPM-KKT}), or not worse than the computation of an OSM per iteration. For the second step, we complete $\Delta y_b^k$ by setting $\Delta y_j = 0$ for any index $j \notin \mathcal{B}$. If an IPM is used, we compute $\Delta y_+^k, \Delta y_-^k$ via $\Delta y_b^k$ as what we have shown in Section~\ref{subsection-correction-IPM}, and $(x^k + \Delta x^k,y^k,y_+^k + \Delta y_+^k,y_-^k + \Delta y_-^k, z^k + \Delta z^k)$ is a dual feasible point for early termination.
\subsection{Applications in Control}\label{subsection:MPC-application}
A common type of MIP arising in control engineering is optimal control with discrete-valued inputs as encountered in hybrid MPC problems, which takes the form:
\begin{align}
	\begin{aligned}
		\min \ & \sum_{t=0}^{T-1} (x_t^\top Q_t x_t + u_t^\top R_t u_t) + x_T^\top Q_T x_T + 2q_T^\top x_T\\
		\text{s.t.} \quad & x_{t+1} = \bar{A}x_t + \bar{B}u_t, \ x_0= x_{init}, \\
		&  u_t \in \mathcal{U}_t, \quad \forall t = 0,1, \dots,T-1,
	\end{aligned}\label{model-hmpc}
\end{align}
where $x_{init} \in \mathbb{R}^{n_x}$ is the initial state and system dynamics is $x_{t+1} = \bar{A}x_t + \bar{B}u_t$ with constraints $\mathcal{U}_t$ for each input $u_t \in \mathbb{R}^{n_u}$.  $\mathcal{U}_t$ can be composed wholly or in part by discrete valued constraints.
Our optimization-based correction is suitable for the hybrid-MPC~\eqref{model-hmpc} as due to the next theorem.
\begin{theorem}
	The optimization-based correction is applicable to the hybrid MPC~\eqref{model-hmpc} when $\mathcal{U}_t$ is bounded for $t=0,1,\dots,T-1$.
\end{theorem}
\begin{proof}
\noindent Suppose $x:=[x_0;\dots;x_T;u_0; \dots;u_{T-1}]$. The corresponding block components of $[P, I_{\mathcal{B}}^\top, G^\top]$ become
\begin{subequations}
	\small
	\begin{align*}
		P = \begin{bmatrix}
			Q & \\
			& R
		\end{bmatrix},
		I_{\mathcal{B}} = \begin{bmatrix}
			0_{n_uT \times n_x(T+1)}  & I_{n_uT}
		\end{bmatrix},  	
	\end{align*} 
	\begin{small}
		\begin{align*}
			G = \begin{bmatrix}
				\begin{matrix}
					I & & & & & \\
					\bar{A} &-I & & & & \\
					& \hphantom{-}\bar{A} & -I& & & \\
					& & \ddots& \ddots& & \\
					& & & &  \hphantom{-}\bar{A}& -I
				\end{matrix}
				~~\vline \quad 
				\begin{matrix}
					0 & & & \\
					\bar{B} & & &\\
					& \ddots& & \\
					& & \ddots& \\
					& & & \bar{B}
				\end{matrix}
			\end{bmatrix},
		\end{align*}
	\end{small}
\end{subequations}
\noindent where $Q = \text{diag}(Q_1, \dots, Q_T), R = \text{diag}(R_1, \dots,R_T)$ are block diagonal. Hence, $[P, I_{\mathcal{B}}^\top, G^\top]$ can be reordered as an upper triangular matrix that looks like
\begin{align*}
\begin{bmatrix}
\begin{matrix}
	I_{n_x}& & \hphantom{-}\bar{A}^\top& & & & \\
	& I_{n_u} & \hphantom{-}\bar{B}^\top& & & &  \\
	& & -I_{n_x}& & & &\\
	& & & & \!\!\!\!\!\!\ddots &  & \hphantom{-}\bar{A}^\top\\
	& & & &  & I_{n_u}& \hphantom{-}\bar{B}^\top\\
	& & & &  & & -I_{n_x} 
\end{matrix} ~~\vline \quad 
\begin{matrix}
	\vdots & \vdots \\
	\vdots & \vdots \\
	\vdots & \vdots \\
	\vdots & \vdots 
\end{matrix}
\end{bmatrix}.
\end{align*}	 
The matrix above is full row rank since the diagonal term is either $1$ or $-1$. Hence, $[P, I_{\mathcal{B}}^\top, G^\top]$ is full row rank and the Assumption~\ref{assumption-2} is satisfied.
\end{proof}
Note that the system~\eqref{correction-partial-linear-system} is also banded for the sparse formulation~\eqref{model-hmpc} and we can exploit its structure to accelerate the computation as in~\cite{Rao98}, which reduces the cost per iteration from $\mathcal{O}((m_xN)^3)$ to $\mathcal{O}(N(n_x+n_u)^3)$.

\section{Algorithm and Complexity of Computation}
We next summarize how to implement early termination in a B\&B, which corresponds to steps $3$-$17$ in Algorithm~\ref{alg_BB}. For every iteration $k$ in a node CP$(\underline{x},\bar{x})$, we can obtain a primal-dual iterate $(x^k,s^k,y^k,y_b^k,z^k)$ from an OSM or an IPM with an approximate dual cost $D^k$. Note that this iterate is conic feasible but doesn't satisfy the dual linear constraint, i.e.~\eqref{ADMM-dual} or~\eqref{IPM-dual}. We then check whether the algorithm finds an optimal solution $\hat{x}$ or detects the infeasibility of CP$(\underline{x},\bar{x})$ (steps $5$-$10$). These steps are inherent to a primal-dual algorithm even without early termination and do not incur any additional time cost. We then activate early termination when we find the approximate dual cost is larger than the current upper bound, i.e $D^k \ge U$ (step $11$). This heuristic follows~\cite{Liang21} since $D^k$ is close to the optimal solution of CP$(\underline{x},\bar{x})$ when the dual linear residual $r^k$ is small enough, and can save computation time on early termination. 

Once early termination is enabled, we then compute a feasible correction $(\Delta x^k, \Delta y_b^k,\Delta z^k)$ using either of the methods discussed in Section~\ref{subsection-correction-splitting} to \ref{subsection:optimization-based-correction-partial} and compute the dual cost $\underline{D}^k$ at the dual feasible point $(x^k + \Delta x^k,s^k,y^k,y_b^k + \Delta y_b^k,z^k + \Delta z^k)$ (step $12$). If $\underline{D}^k$ is greater than $U$, we know the optimum of CP$(\underline{x},\bar{x})$ is larger than $\underline{D}^k$ due to weak duality, and hence larger than $U$, which indicates that we can stop the node computation and prune this node immediately. Otherwise, we continue computing until we solve CP$(\underline{x}, \bar{x})$ and then proceeds with the standard B\&B ,method (steps $18$-$27$). 
\begin{breakablealgorithm}
	\caption{B\&B for MICP with early termination}
	\label{alg_BB}
	\begin{algorithmic}[1]
		\REQUIRE ~~\\  
		Initialize upper bound $U \gets + \infty$, node tree $\mathcal{T} \gets \text{CP}(l,u)$ \\[7pt]
		\WHILE {$\mathcal{T} \ne \emptyset$}
		\STATE Pick and remove CP$(\underline{x}, \bar{x})$ from $\mathcal{T}$ 
		\FOR {k = 1, 2 \dots}
		\STATE Generate $(x^k,s^k,y^k,y_b^k,z^k)$ and an estimated dual cost $D^k$ from OSMs or IPMs
		\IF {\textit{termination criteria is satisfied}}
		\STATE return optimal solution $\hat{x} = x^{k}$ and $f(\hat{x})$
		\ENDIF
		\IF {\textit{infeasibility of} \text{CP}$(\underline{x}, \bar{x})$ \textit{is detected}}
		\STATE return CP$(\underline{x}, \bar{x})$ infeasible 
		\ENDIF
		\IF {$D^k \ge U$}
			\STATE Compute the corrected dual cost $\underline{D}^k$ via $(x^k + \Delta x^k,y^k,y_b^k + \Delta y_b^k,z^k + \Delta z^k)$
			\IF {$\underline{D}^k \ge U$}
			\STATE return CP$(\underline{x}, \bar{x})$ terminates early
			\ENDIF
		\ENDIF 
		\ENDFOR 
		\IF {CP$(\underline{x}, \bar{x})$ \textit{terminates early} or is \textit{infeasible}}
		\STATE prune current node
		\ELSIF {$f(\hat{x}) > U$}
		\STATE prune current node
		\ELSIF {$\hat{x}$ is \textit{integer feasible}}
		\STATE $U \gets f(\hat{x}), x^* \gets \hat{x}$
		\STATE prune nodes in $\mathcal{T}$ with lower bound $> U$
		\ELSE
		\STATE branch node CP$(\underline{x}, \bar{x})$
		\ENDIF
		\ENDWHILE
	\end{algorithmic}
\end{breakablealgorithm}

Suppose we already have a dual cost $D^k$ based on the iterate $(x^k,y^k,y_b^k,z^k)$ from a primal-dual algorithm. In that case the correction~\eqref{box-correction} only takes extra $O(n)$ flops to generate a feasible dual cost. For an optimization-based correction~\eqref{correction-opt-partial}, we need no more than $O((2n+p)^2)$ flops to solve the linear system~\eqref{correction-partial-linear-system} if we save the factorization of the matrix in~\eqref{correction-partial-linear-system} from the start of a MICP. Both correction flops are relatively small compared to $O(n+p+m)^3$ flops per IPM iteration. For OSM, we check early termination along with the termination check, which is usually repeated every $M=25$ iterations. Each early termination check is no more costly than the original computation in one iteration, so that its computational time is negligible inside every $M$ iterations. 

\section{Numerical Results}
We implement Algorithm~\ref{alg_BB} and a counterpart without early termination, i.e. removing steps $3$-$17$ in Algorithm~\ref{alg_BB}. Both were written in Julia with every convex relaxation solved by the IPM solver Clarabel~\cite{Clarabel}. Tests are implemented on Intel Core i7-9700 CPU @3.00GHz, 16GB RAM.

\subsection{Mixed Integer Model Predictive Control}
We consider a hybrid MPC for current reference tracking from~\cite{Stellato17}, which can be formulated as a MIQP
\begin{align}
\begin{aligned}
	\min \quad & \sum_{t=0}^{T} \gamma^t l(x_t) + \gamma^{T} V(x_T) \\
	\text{s.t.} \quad & x_0 = x_{\text{init}}, \\
	& x_{t+1} = \bar{A}x_t + \bar{B}u_t, \\
	& ||u_t - u_{t-1}||_{\infty} \le 1, \ u_t \in \{-1,0,1\}^6,
\end{aligned}\label{model-hmpc-experiment}
\end{align}
where $\gamma$ is a discount factor and $T$ is the time horizon. The quadratic state penalty cost $l(x_t)$ is for current tracking and $V(x_T)$ is a final stage cost using approximate dynamic programming. The initial state is $x_{\text{init}}$ and the system dynamics is $x_{t+1} = \bar{A}x_t + \bar{B}u_t$ with $x_t \in \mathbb{R}^{12}$ representing the internal motor currents, voltages and the input $u_t \in \mathbb{R}^{6}$ including three semiconductor devices positions with integer values $\{-1,0,1\}$ and three additional binary components required to model the system. The ramp rate constraint $\|u_t - u_{t-1}\|_{\infty} \le 1$ avoids shoot-through in the inverter positions  (changes from $-1$ to $1$ or vice-versa) that can damage the components. 

By eliminating $x_t, t \in \{1, \dots, T\}$ via the state dynamics, problem \eqref{model-hmpc-experiment} reduces to a problem depending only on input variables $u_0, \dots, u_{T-1}$ and the initial state $x_0$; we refer readers to~\cite{Stellato17} for details. We set $T=8$ for the time horizon and simulate closed-loop MIMPC for $100$ consecutive intervals. Figure~\ref{fig:mpcn=8-denseMPC} compares the performance of B\&B with and without early termination. We run the test with both cold-start and warm-start to initialize the solver variables using the solution from the parent subproblem in the B\&B tree. We take the simple early termination introduced in Section~\ref{subsection-correction-IPM}. Since a valid upper bound $U$ is required for early termination, we start to count IPM iterations only when the first feasible solution of~\eqref{model-hmpc-experiment} is found. Here, we define one loop of steps 3-17 in Algorithm \ref{alg_BB} as an IPM iteration. For all 100 intervals, early termination has produced a noticeable reduction in IPM iterations, averaging to about $25\%$. Since the simple early termination only takes additional $O(n)$ flops compared to the factorization with $O(n^3)$ flops per iteration, the ratio of reduction of total iteration numbers is a good proxy for the ratio of solve time reduction we can achieve in IPMs when a simple early termination is implemented. 
\begin{figure}[ht]
	\centering
	\vspace{-4mm}
	\includegraphics[width=0.9\linewidth]{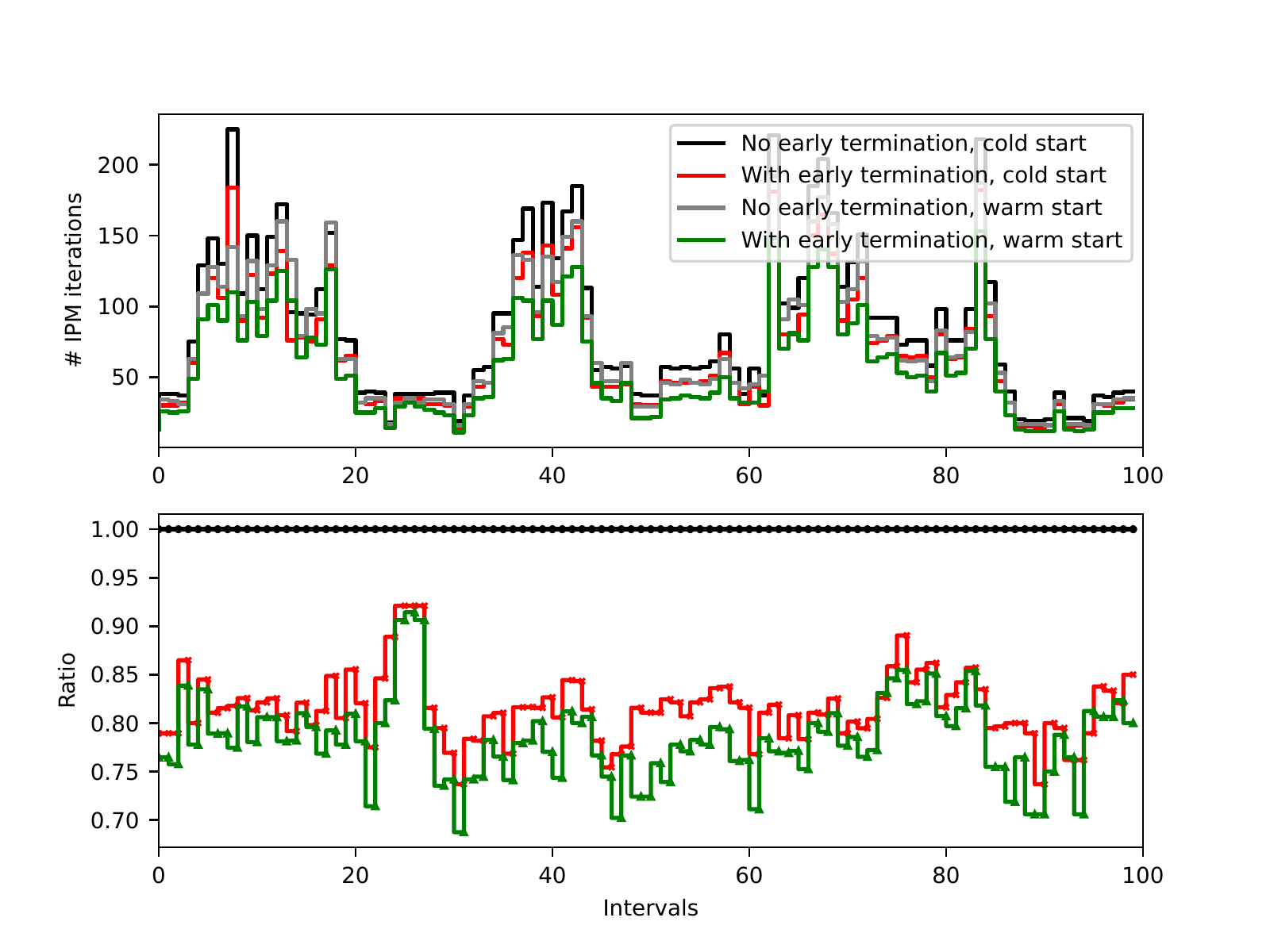}
	\caption{\footnotesize{MIMPC $T = 8$, reduced dense form}}
	\label{fig:mpcn=8-denseMPC}
	\vspace{-3mm}
\end{figure}

We then implement another experiment for~\eqref{model-hmpc-experiment} but use the non-reduced sparse form as discussed in Section~\ref{subsection:MPC-application} with the optimization-based correction of $\eta = \gamma = 1$. The optimization-based correction takes an additional $O(n^2)$ flops compared to the factorization with $O(n^3)$ flops per iteration, so the ratio of reduction of total iteration numbers remains a good proxy for the ratio of solve time reduction we can achieve in IPMs when a simple early termination is implemented. Although early termination seems to be somewhat less effective in the non-reduced form of MPC relative to the dense form, it arises from the fact that we have fewer IPM iterations left to go once the first feasible upper bound $U$ is found, which is shown in Figure~\ref{fig:mpcn=8-denseMPC} and Figure~\ref{fig:mpcn=8-sparseMPC}.
\begin{figure}[ht]
	\centering
	\vspace{-4mm}
	\includegraphics[width=0.9\linewidth]{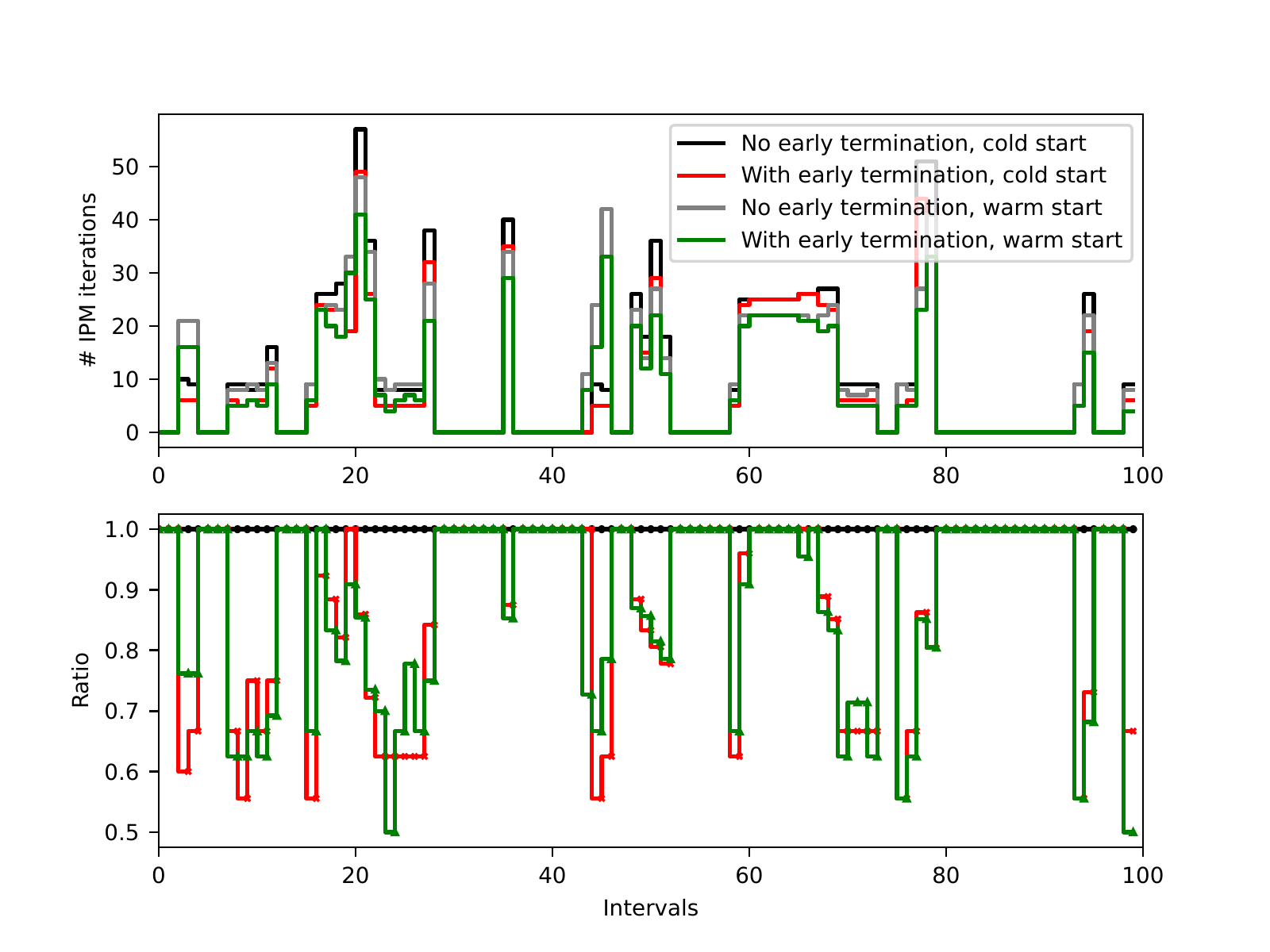}
	\caption{\footnotesize{MIMPC $T = 8$, non-reduced sparse form}}
	\label{fig:mpcn=8-sparseMPC}
	\vspace{-3mm}
\end{figure}

\subsection{Portfolio Optimization}
We also test our proposed early termination technique on a portfolio optimization, which can be formulated as a mixed integer second-order cone (SOC) programming~\cite{portfolio},
\begin{align*}
	\begin{aligned}
		\min \quad & r^\top (x^{+}-x^{-}) \\
		\text{s.t.} \quad & (x^{+}-x^{-})^{\top}\Lambda(x^{+}-x^{-}) \leq \rho, \\
		& \sum\nolimits^{n}_{i=1}(x^{+}-x^{-})=1, \quad \sum\nolimits^{n}_{i=1}b_i \leq K,\\
		&L_{min}\leq \sum\nolimits^{n}_{i=1}l_i \leq L_{max}, \  b \leq Hl, \ l \leq H^{\top}b,\\
		& l_j \in \{0,1\}, \text{ for } j\in  \{1,...,L\}\\
		& 0 \leq x_i^+ \leq b_i, 0 \leq x_i^- \leq b_i, b_i \in \{0,1\}, i \in \{1,...,n\}.
	\end{aligned}  
\end{align*}
There are $n$ assets in total, categorized into $L$ industry sectors, with the mapping from assets to sectors captured by matrix $H \in \mathbb{R}^{n \times L}$. We define $x=(x^{+}; x^{-}) \in \mathbb{R}^{2n}$ as the fractions of portfolio value held in each asset: $x^{+}$ and $x^{-}$ denote buying and selling (i.e.\ shorting) respectively, and both are non-negative and must sum up to unity. $r \in \mathbb{R}^{n}$ is the expected return for $n$ assets, and $\Lambda$ is the covariance for market volatility and restricted below a certain level $\rho$ and formulated as a SOC constraint. The binary vectors $b\mathbb{R}^{n}$ and $l \in \mathbb{R}^L$ denote whether we invest in an asset, respectively in a sector or not. The number of assets we can invest in is upper-bounded by $K$ and the number of sectors is box-constrained by $L_{min}$ and $L_{max}$ to ensure asset diversity. 

We use the early termination strategy as in Section~\ref{subsection-correction-IPM} and choose $n = 20, L=3$, $T=2000, L_{min}=1, L_{max}=L, \rho=100, K = 10$. Figure~\ref{fig:portfolio-optimization} shows the early termination can reduce about $10\%$-$15\%$ of IPM iterations after we find the first integer feasible solution, which proves that our early termination remains effective for general MICPs.

\section{Conclusion}
We generalized our early termination technique of ADMM in~\cite{Chen22} to state-of-the-art primal-dual algorithms in MICPs. We showed how to utilize existing dual iterates inside either an OSM or an IPM to generate a dual feasible point for early termination with little additional efforts, and we provided a sufficient condition when we can find a dual feasible point in the proposed early termination technique. Numerical results showed the proposed early termination can reduce the total number of iterations in MICPs effectively.
\begin{figure}[ht]
	\centering
	\vspace{-4mm}
	\includegraphics[width=0.9\linewidth]{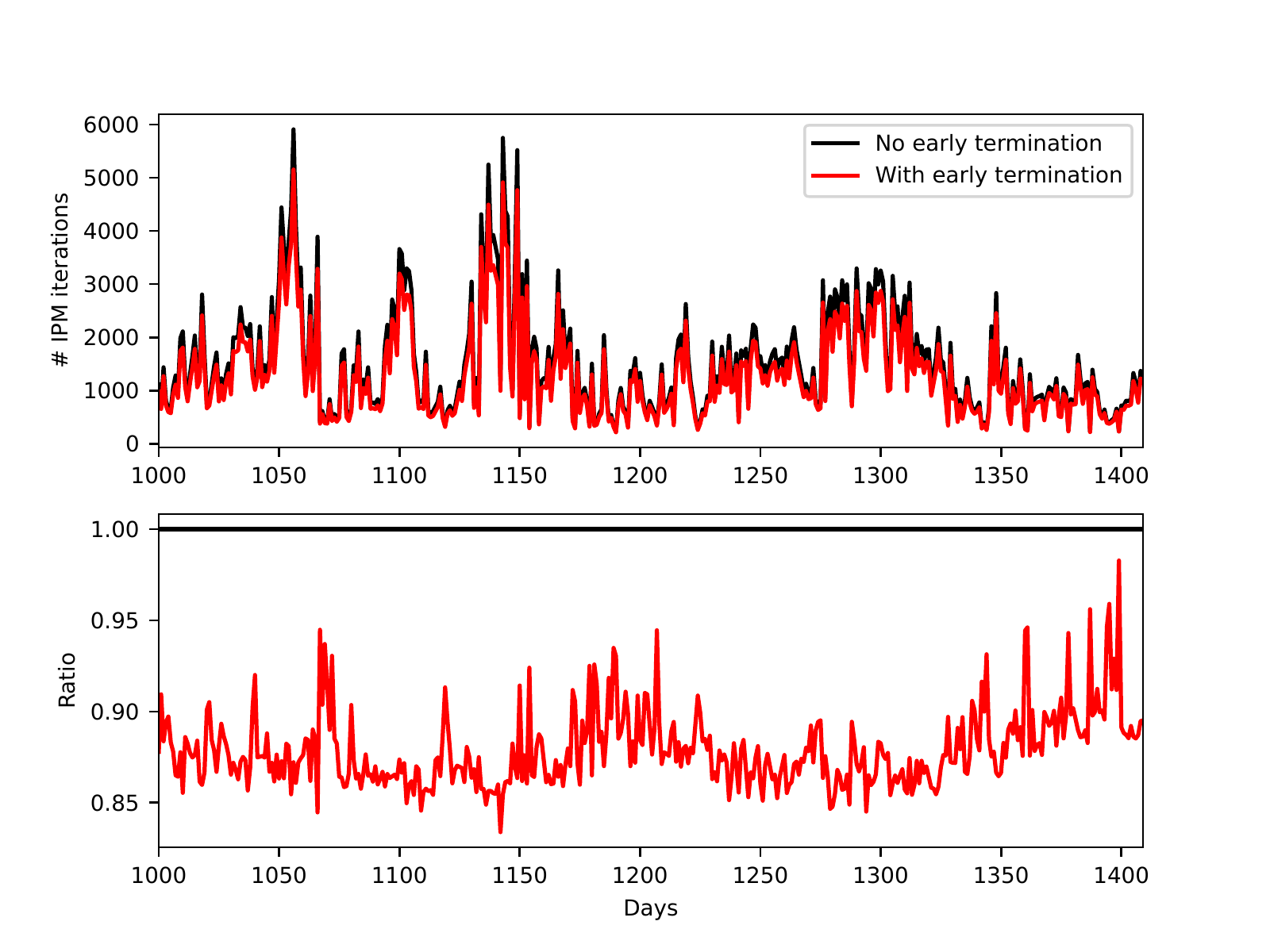}
	\caption{\footnotesize{Portfolio Optimization}}
	\label{fig:portfolio-optimization}
	\vspace{-3mm}
\end{figure}

\normalem
\bibliographystyle{IEEEtran}
\bibliography{reference}

\end{document}